\documentclass[12pt]{article}
\usepackage{amsmath}
\usepackage{amsthm}
\usepackage{amsxtra}
\usepackage{amsfonts,amssymb}
\usepackage{graphicx}
\usepackage{marginnote}
\usepackage[top=1.5in, bottom=1.5in, left=1in, right=1in]{geometry}

\newtheorem{Th}{Theorem}
\newtheorem{Prop}[Th]{Proposition}

\newtheorem{Co}[Th]{Corollary}

{\bf}{\it}

\def\mn{\mathbb{N}}

\def\mq{\mathbb{Q}}

\def\mc{\mathbb{C}}

\def\mrc{\mathbb{R}_\mathcal{C}}

\newtheorem{Def}[Th]{Definition}

   \def\mn{\mathbb{N}}

\def\Arg{\text{Arg}}

\newcommand{\ie}{\emph{i.e.}\ }

\begin{document}
\title{On computational complexity of Cremer Julia sets.}
\author{Artem Dudko\thanks{A. Dudko acknowledges the support by the National Science Centre, Poland,
grant 2016/23/P/ST1/04088 under POLONEZ programme which has received funding from the EU\;\protect\includegraphics[width=.03\linewidth]{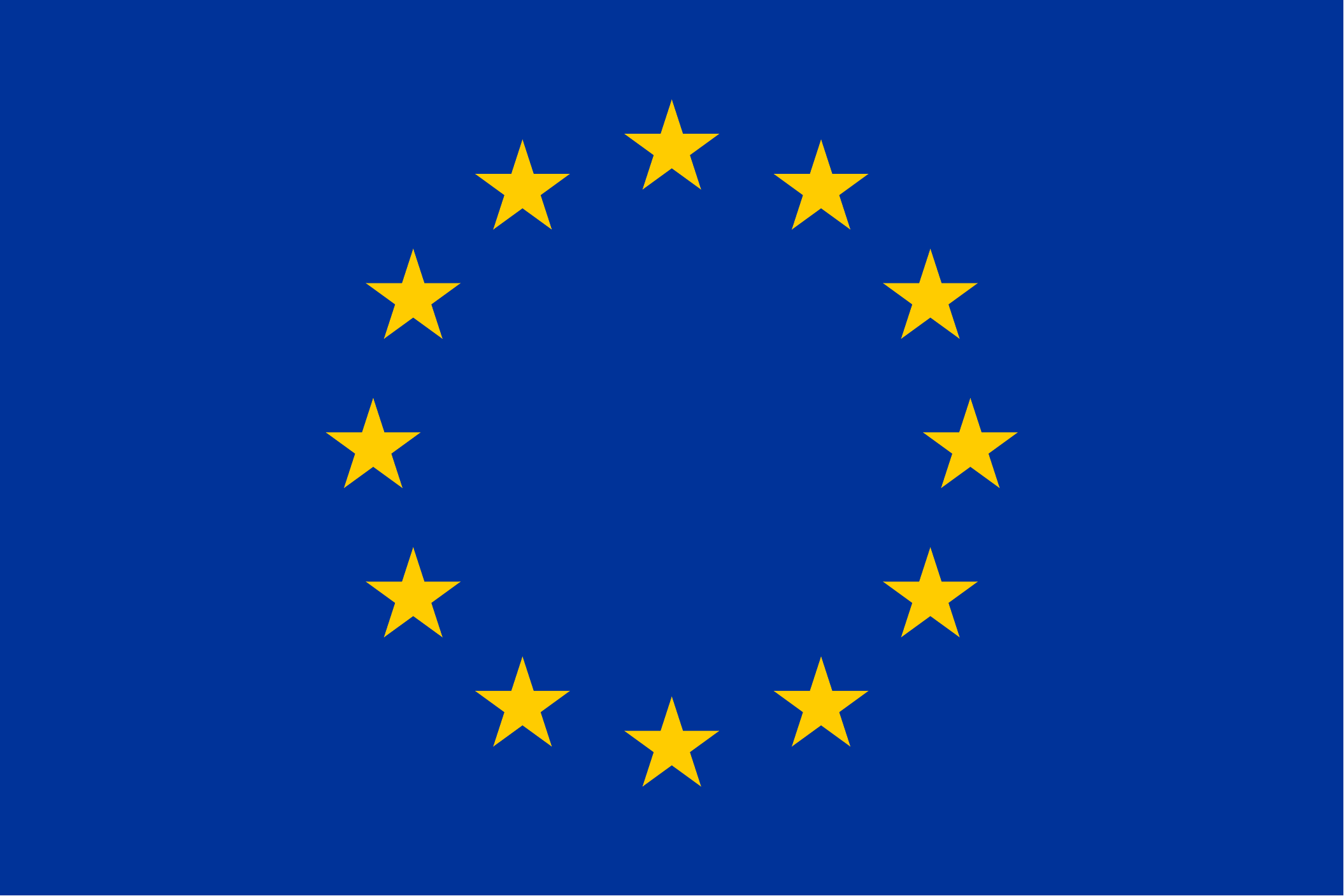} Horizon 2020 research and innovation programme under the MSCA grant agreement No. 665778} and Michael Yampolsky\thanks{M. Yampolsky was partially supported by NSERC Discovery Grant}}
\date{}

\abovedisplayskip 8pt \abovedisplayskip 8pt \belowdisplayskip 9pt

\maketitle

\begin{abstract}
	We find an abundance of Cremer Julia sets of an arbitrarily high computational complexity.
\end{abstract}


\section{Introduction.}
Most of us have seen pictures of a quadratic Julia sets on a computer screen. A program to visualize such a set seems easy to write based on its definition. Let us start by noting that a linear change of coordinates transforms every quadratic polynomial into the form
$$p_c(z)=z^2+c,$$
and no two functions $p_c(z)$ with different values of $c$ are linearly conjugate. Thus when we study the dynamics of quadratics, it is sufficient to restrict ourselves to maps of this form.

 If we iterate $p_c$ starting with a point $z_0\in\Bbb C$,
we will obtain an infinite orbit
$$z_0,\;z_1=p_c(z_0),\;z_2=p_c(z_1)=p_c(p_c(z_0)),\ldots,z_n=p_c^{\circ n}(z_0),\ldots$$
Clearly, if $z_0>>1$ is large enough in relation to $c$, then $z_1\approx z_0^2$, $z_n\approx z_0^{2^n}$, and so $z_n\to\infty$. Thus, the set $K_c$ consisting of initial points $z_0$ whose orbits {\it do not} converge to infinity is bounded. It is known as the {\it filled Julia set} of $p_c$, and the Julia set $J_c$ is defined as its boundary:
$$J_c=\partial K_c.$$

Alternatively, we can  define $J_c$ as the {\it repeller} of the dynamical system $p_c:\mathbb C\to\mathbb C$. That is, $J_c$ is the limit set of the inverse images $p_c^{-n}(z_0)$ for all $z_0\in\mathbb C$ except at most one value (the unique exception happens when $c=0$ and $z_0=0$).

This definition suggests what is perhaps the simplest approach to computing $J_c$: use the set set $p_c^{-n}(z_0)$ for some $z_0\in\mathbb C$ and a very large $n$ to approximate $J_c$. Often, a program like this produces a satisfactory image. Its principal shortcoming, however, is evident: how can we tell what $n$ to choose to
get an approximation of the picture of $J_c$ (more formally, for which $n$ will the distance between
$J_c$ and $p_c^{-n}(z_0)$ be less than one pixel size for the desired screen resolution)?

An alternative approach, based on the first definition is also common: iterate points in the plane (centers of pixels with the given screen resolution) $n$ times, and see if the modulus of an iterate exceeds some fixed bound $M=M(c)$ such that, for instance, $|p_c(z)|>2|z|$ for $|z|>M$. Remove such points from the picture; what is left is an approximation of $K_c$, and $J_c$ is its boundary. A similar problem arises here: how do we know what $n$ to choose, as some points $z$ whose orbits escape to $\{|z|>M\}$ may take an arbitrarily long time to do so?

The above questions are two instances of the celebrated Halting Problem, which is an example of an algorithmically unsolvable problem given by Turing in \cite{Tur}. In fact, M.~Braverman and the second author have shown that for some Julia sets obtaining a faithful computer image is {\it as hard as} the Halting Problem (see \cite{BY06,BY08}). So those are pictures that we can never hope to see. However, in a way, the {\it non-computable} examples of Julia sets are well understood. They belong to a class of Siegel quadratic polynomials
(see \S~\ref{sec:cremer} for the definition), and at least for some of them (the locally connected ones produced in  \cite{BY09}) we know what they would look like.

There is another class of Julia sets that has long baffled computational practitioners: Cremer quadratic Julia sets. We discuss the definition of such Julia sets in \S~\ref{sec:cremer}. We do not know what their picture would look like and, so far, no one has been able to produce an informative image of such $J_c$. Counter-intuitively, all Cremer quadratic Julia sets are computable, at least in theory. This was proven in \cite{BBY}, where an explicit algorithm for computing accurate images of Cremer quadratic Julia sets was given. The algorithm is, however, not practical. Its running time on a screen with a reasonable resolution would be enormous. However, its existence allows us to formulate the following questions:
\begin{itemize}
\item[(I)] Does there exist an algorithm to compute at least one Cremer Julia sets with a practical running time (for instance, polynomial in $n$, where $2^{-n}$ is the size of the pixel on the computer screen)?
\item[(II)] Does there exist at least one Cremer Julia set for which every algorithm will have an impractical running time (i.e. can we prove that there is at least one such $J_c$ with a high, for instance, non-polynomial, lower complexity bound)?
\end{itemize}
Note that the two questions are not mutually exclusive.  In the present paper we show that the second one has an emphatically positive answer:

\medskip
\noindent
{\sl For every lower bound $t(n)$ there exist an abundance of Cremer quadratic Julia sets whose computational complexity is not lower than $t(n)$}.

\medskip
\noindent
The structure of the paper is as follows. In \S~\ref{sec:cremer} we briefly review the definitions of Cremer Julia sets. In \S~\ref{sec:comput} we formalize the concept of computational complexity of $J_c$, and in \S~\ref{sec:result} we formulate our main theorem. In \S~\ref{sec:lavaurs} we introduce the main tools of Complex Dynamics used in the proof. We present the proof in \S~\ref{sec:proof}.

\subsection{Cremer quadratic Julia sets}
\label{sec:cremer}

We refer the reader to the classical book of Milnor \cite{Mil} for a detailed introduction to the basic concepts of Complex Dynamics. We will assume familiarity with the standard definitions, and will only briefly recall a few  facts about Cremer quadratics below.

 Let $f$ be a holomorphic map defined on an open domain $U$. Let $z_0\in U$ be a periodic point of period $p$ for $f$. Denote by $\lambda=Df^p(z_0)$ the multiplier of $z_0$.
Fatou-Shishikura bound implies that a quadratic polynomial $p_c$ can have at most one periodic point whose multiplier $\lambda\in\{|z|\leq 1\}$. We are principally interested in the case of irrationally indifferent periodic points, that is, $\lambda=e^{2\pi i\theta}$ with $\theta\in\mathbb R \setminus \mathbb Q$.

The map $f^p$ is called linearizable on a neighborhood of $z_0$ if there exist a neighborhood $V$ of $z_0$ and a conformal map $\phi$ from $V$ to a neighborhood of $0$ such that $$\phi\circ f^p\circ\phi^{-1}(z)=\lambda z\;\;\text{for all}\;\;z\in\phi(V).$$
 The point $z_0$ with $\lambda=\exp(2\pi i\theta)$, $\theta\in\mathbb R\setminus\mathbb Q$ is called a Cremer periodic point if  $f^p$ is not linearizable near $z_0$; it is called a Siegel point otherwise.

 It is known that the property of being Cremer is directly related to the Diophantine properties of $\theta\notin\mathbb Q$.  Namely, let $$\theta=\frac{1}{a_1+\frac{1}{a_2+\frac{1}{a_3+\ldots}}}, \;a_i\in \mathbb N,\;\text{ and let }\;\frac{p_n}{q_n}=\frac{1}{a_1+\frac{1}{a_2+\ldots\frac{1}{a_n}}}$$ be the $n$-th continued fraction convergent  of $\theta$. Brjuno \cite{Brjuno-71} showed that if the following condition is satisfied
$$\sum\limits_{n=1}^\infty \frac{\log q_{n+1}}{q_n}<\infty $$ ($\theta$ is called \emph{Brjuno number} in this case) then the map $f^{ p}(z)$ is linearizable near $z_0$. 

Let us now specialize to the case of quadratic polynomials. A quadratic map has two fixed points, counted with multiplicity. It will be convenient to us to consider quadratic polynomials of the form
\begin{equation}\label{quadf2}
z\mapsto \lambda z+z^2,
\end{equation}
with a fixed point at the origin, whose multiplier is equal to $\lambda$. A more familiar looking formula
 $p_c(z)=z^2+c$ is transformed into (\ref{quadf2}) with
$$c=\lambda/2-\lambda^2/4$$ by the linear change of coordinates
$$w=z-\lambda/2.$$
Since we are specifically interested in the case $\lambda=e^{2\pi i\theta}\in S^1$, let us set
$$f_\theta(z)=e^{2\pi i\theta}z+z^2\text{, for }\theta\in\mathbb R;$$
this is a one real parameter family of quadratic polynomials.
 Yoccoz \cite{Yoccoz-95} proved a famous converse of Brjuno's Theorem for this family, that is
if $\theta$ is not Brjuno then $0$ is a Cremer point of $f_\theta$.

 In what follows, we will mostly restrict our attention to the family $f_\theta$. Where it does not lead to a confusion, we will write $J_\theta$, $K_\theta$, etc., for the Julia set and the filled Julia set of $f_\theta$.

 Yoccoz showed that quadratic maps $f_\theta$ with a Cremer fixed point have the {\it Small Cycle property}, \ie there are  periodic cycles contained in arbitrarily small neighborhoods of the Cremer fixed point, different from the Cremer point itself \cite{Yoccoz-95}. Vice  versa, if a fixed point of a holomorphic map has the Small Cycle property it is necessarily a Cremer fixed point since by trivial argumentation it can not be of any other type (attracting, repelling, Siegel or parabolic).

\subsection{Preliminaries on computability}
\label{sec:comput}
In this section we briefly recall the notions of computability and computational complexity of sets.
For a more detailed exposition  we refer the reader to the monograph \cite{BY08}.
 The notion of computability relies on the concept of a Turing Machine (TM) \cite{Tur},
 which is a commonly accepted way of formalizing
the definition of an algorithm.
A precise description of a Turing Machine is quite technical and we
do not give it here, instead referring the reader to any text on Computability Theory (e.g. \cite{Pap} and \cite{Sip}).
The computational power of a Turing Machine is provably equivalent to that of a computer program running on a RAM
computer with an unlimited memory.

\begin{Def}\label{comp fun def}
A function $f:\mn\rightarrow \mn$ is called computable, if there exists
a TM which takes $x$ as an input and outputs $f(x)$.
\end{Def}
Note that Definition \ref{comp fun def} can be naturally extended to
functions on arbitrary countable sets, using a convenient
identification with $\mathbb{N}$. The following definition of a computable real number is
due to Turing \cite{Tur}:
\begin{Def} A real number $\alpha$ is called computable if there is a
 computable function $\phi:\mathbb{N}\rightarrow \mathbb{Q}$,
 such that for all $n$ $$\left|\alpha-\phi(n)\right|<2^{-n}.$$
\end{Def}
 The set of computable reals is denoted by $\mathbb{R}_\mathcal{C}$. Trivially, $\mq\subset\mrc$. Irrational numbers
such as $e$ and $\pi$ which can be computed with an arbitrary precision also belong to $\mrc$. However, since there
exist only countably many algorithms, the set $\mrc$ is countable, and hence a typical real number is not
computable.

 The set of computable complex numbers is
defined by
$\mathbb{C}_\mathcal{C}=\mathbb{R}_\mathcal{C}+i\mathbb{R}_\mathcal{C}$.
Note that  $\mathbb{R}_\mathcal{C}$ (as well as $\mathbb{C}_\mathcal{C}$)
considered with the usual arithmetic operation forms a  field.

To define computability of functions of real or complex variable we need to introduce the concept of an oracle:
\begin{Def}
A function $\phi:\mn\to\mq+i\mq$ is an oracle for $c\in\mc$ if for every $n\in\mn$ we have
$$|c-\phi(n)|<2^{-n} .$$
\end{Def}
A TM equipped with an oracle (or simply an {\it oracle TM}) may query the oracle by reading the value of $\phi(n)$ for an arbitrary $n$.
\begin{Def}
Let $S\subset \mc$. A function $f:S\to \mc$ is called computable if there exists an oracle TM $M^\phi$ with a single natural
input $n$ such that if $\phi$ is an oracle for $z\in S$ then $M^\phi$ outputs $w\in \mq+i\mq$ such that
$$|w-f(z)|<2^{-n} .$$
\end{Def}
We say that a function $f$ is {\it poly-time computable} if in the above definition the algorithm $M^\phi$ can be made to run
in time bounded by a polynomial in $n$, independently of the choice of a point $z\in S$ or an oracle representing this
point. Note that when calculating the running time of $M^\phi$, querying $\phi$ with precision $2^{-m}$ counts as
$m$ time units. In other words, it takes $m$ ticks of the clock to read the argument of $f$ with precision $m$ (dyadic)
digits.

 Let $d(\cdot,\cdot)$ stand for Euclidean distance between points or sets in $\mathbb{R}^2$.
 Recall the definition of the {\it Hausdorff distance} between two sets:
$$d_H(S,T)=\inf\{r>0:S\subset U_r(T),\;T\subset U_r(S)\},$$
where $U_r(T)$ stands for the $r$-neighborhood of $T$:
$$U_r(T)=\{z\in \mathbb{R}^2:d(z,T)\leqslant r\}.$$ We call
a set $T$ a {\it $2^{-n}$ approximation} of a bounded set $S$, if
$d_H(S,T)\leqslant 2^{-n}$. When we try to draw a $2^{-n}$ approximation $T$ of a set $S$
 using a computer program, it is convenient to
 let $T$ be a finite
 collection of disks of radius $2^{-n-2}$ centered at points of the form $(i/2^{n+2},j/2^{n+2})$
 for $i,j\in \mathbb{Z}$.  We will call such a set {\it dyadic}.
A dyadic set $T$ can be described using a function
\begin{eqnarray}\label{comp fun}h_S(n,z)=\left\{\begin{array}{ll}1,&\text{if}\;\;d(z,S)\leqslant 2^{-n-2},
\\0,&\text{if}\;\;d(z,S)\geqslant 2\cdot 2^{-n-2},\\
0\;\text{or}\;1&\text{otherwise},
\end{array}\right.\end{eqnarray} where $n\in \mathbb{N}$ and $z=(i/2^{n+2},j/2^{n+2}),\;i,j\in \mathbb{Z}.$\\
 \\ Using this
function, we define computability and computational
complexity of a set in $\mathbb{R}^2$ in the following way.
\begin{Def}\label{DefComputeSet} Let $t:\mathbb N\to \mathbb N$. A bounded set $S\subset
\mathbb{R}^2$ has time complexity bounded by $t(n)$ if there
exist $n_0\in\mathbb N$ and a TM, which computes values of a function $h(n,\bullet)$
of the form (\ref{comp fun}) in time $t(n)$, for all $n\geq n_0$. We say that
$S$ is poly-time computable, if there exists a polynomial
$p(n)$, such that $S$ is computable in time $p(n)$.
\end{Def}

\subsection{The main result}
\label{sec:result}
\begin{Th}\label{ThMain} For any function $t:\mathbb N\to \mathbb N$ there exists a dense $G_\delta$ subset of Cremer parameters $S_t\subset S^1$ such that for any $\theta\in S_t$ the Julia set of $f_\theta$ has time complexity not lower than $t(n)$.
\end{Th}
\noindent Let us explain the meaning of the statement of Theorem \ref{ThMain} in more detail. Given a function $t(n)$ and a parameter $\theta\in S_t$ the map $f_\theta$ has a Cremer fixed point. The Theorem states that for any Turing machine $M=M(n)$ with an oracle for $\theta$ there exists a sequence of integers $n_i\to\infty$ such that $M(n_i)$ will not produce a correct $2^{-n_i}$-approximation of $J_\theta$ in time less or equal to $t(n_i)$.
\section{Lavaurs maps}
\label{sec:lavaurs}
Our exposition of Douady-Lavaurs theory \cite{Lav} of parabolic implosion follows the lecture notes \cite{Zin} (see, in particular Theorem~2.3.2 there).

Let $\theta=p/q$. Then $f_\theta$ has a parabolic fixed point at the origin with exactly $q$ attracting and $q$ repelling directions. Denote these directions by $\nu_a^i$ and $\nu_r^i$, $i=0,\ldots,q-1$, so that $\nu_{a/r}^{i+1\mod q}$ is obtained from $\nu_{a/r}^i$ by rotating by the angle $2\pi /q$ and $\nu_r^i$ is obtained from $\nu_a^i$ by rotating by the angle $\pi/q$. Fix corresponding $q$ attracting and $q$ repelling petals $P_a^i,P_r^i$. Let $\Phi_a^i$ and $\Phi_r^i$ be the attracting and the repelling Fatou coordinates for $f_\theta^q$ defined on the unions of these petals, \ie
$$\Phi_{a/r}^i(f_\theta^q)(z)=\Phi_{a/r}^i(z)+1,\;\;z\in P_{a/r}^i.$$ The interior of the filled Julia set $K_\theta$ can be written as a disjoint union
$$K_\theta= \sqcup_{i=0}^{q-1} U_i\;\; \text{for}\;\; U_i=\{z:f_\theta^{qr}(z)\in P_a^i\;\;\text{for some}\;\;r\geqslant 0\}.$$ Fix $i$. Extend $\Phi_a^i$ to $U_i$ by $$\Phi_a^i(z)=\Phi_a^i(f_\theta^{qn}(z))-n$$ whenever $f_\theta^{qn}(z)\in P_a^i$. The inverse of the repelling Fatou coordinate $\Phi_r^i$ extends to a holomorphic map $\Psi_r^i$ on $\mathbb C$.

For $\sigma\in\mathbb C$ let $T_\sigma(z)=z+\sigma$ be the shift map. The {\it Lavaurs map} $L_\sigma$ is defined on $U_i$ by
$$L_\sigma(z)=\Psi_r^i\circ T_\sigma\circ \Phi_a^i.$$ Notice that each Fatou coordinate is defined up to an additive constant. Changing this constant transforms $L_\sigma$ into $L_{\sigma+\tau}$ for some $\tau\in\mathbb C$.
\begin{Th}\label{ThDouady} Let $\theta=p/q$, $p,q$ coprime. For an appropriate choice of Fatou coordinates for $f_\theta$ the following is true. Assume that $\epsilon_k\to 0$, $|\Arg(\epsilon_k)|<\pi/4$, and $N_k\in\mathbb Z$ are such that \begin{equation}\label{EqEpsN}-\frac{\pi}{\epsilon_k q^2}+N_k\rightarrow \sigma\in\mathbb C,k\to\infty.\end{equation} Then $f_{\theta+\epsilon_k}^{N_k}$ converges uniformly on compact subsets of the interior of $K_{\theta}$ to the map $L_\sigma$.
\end{Th}
Given a rational number  $\theta$ fix Fatou coordinates of the corresponding parabolic map $f_\theta$ as in Theorem \ref{ThDouady}.
For a complex number $\sigma$ the {\it filled Lavaurs Julia set} is defined as follows:
$$K_{\theta,\sigma}=\overline{\{z\in K_\theta:L_\theta^n(z)\in K_\theta\;\;\forall\;\;n\in\mathbb N \}},$$
and the {\it Lavaurs Julia set} is it's boundary:
$$J_{\theta,\sigma}=\partial K_{\theta,\sigma}.$$
Using Lavaurs maps, Douady \cite{Douady-94} showed that the correspondence $$\theta\mapsto J_\theta$$ is discontinuous with respect to the Hausdorff metric on compact sets at parabolic parameters (\ie rational $\theta$). In particular, given a rational $\theta$, $\sigma\in\mathbb C$ and a sequence $\epsilon_n$ as in \eqref{EqEpsN} one has:
\begin{equation}\label{EqRelationsJulias}J_\theta\subsetneq J_{\theta,\sigma}\subset\liminf J_{\theta+\epsilon_n}\subset\limsup K_{\theta+\epsilon_n}\subset K_{\theta,\sigma}\subsetneq K_\theta.\end{equation}
For $\sigma\in\mathbb R,\epsilon=\{\epsilon_n\}$ as above assume, in addition, that $\theta+\epsilon_n$ is Cremer for each $n$ and there exists a limit
$\lim\limits_{n\to \infty} J_{\theta+\epsilon_n}$. Denote this limit by $\widetilde J_{\sigma,\epsilon}$. Let $\mathcal J(\sigma)$ be the set of all possible limits $\widetilde J_{\sigma,\epsilon}$.

\begin{figure}[ht]
\centerline{\includegraphics[width=0.7\textwidth]{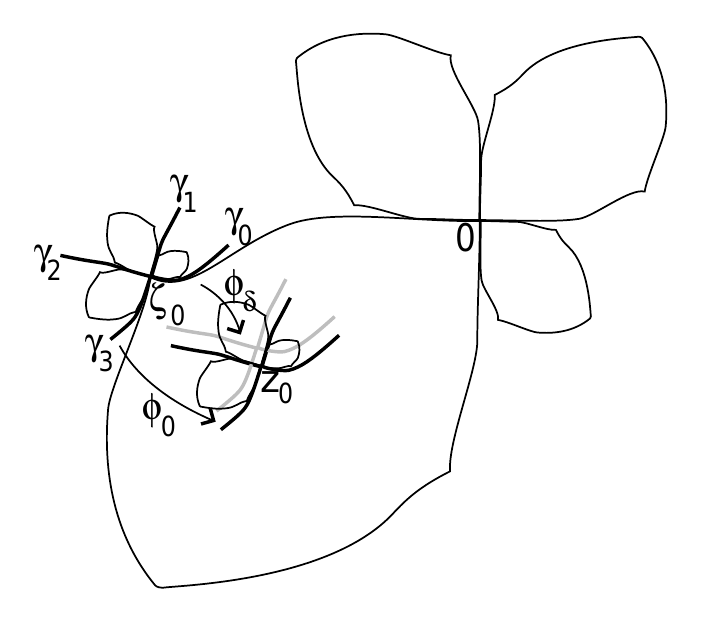}}
\caption{\label{fig1}Illustration to the proof of Proposition~\ref{PropDistLim}}
\end{figure}

\begin{Prop}\label{PropDistLim} For every rational number $p/q$ and every $\sigma_0\in\mathbb R$ there exists at most countably many $\sigma\in\mathbb R$ such that $\mathcal J(\sigma)\cap\mathcal J(\sigma_0)\neq\varnothing$.
\end{Prop}
\begin{proof} See Figure~\ref{fig1} for an illustration. Let $\theta$ be a rational number and $\sigma\in\mathbb C$. Let $\epsilon=\{\epsilon_n\}$ be a sequence as in \eqref{EqEpsN}. Passing to a subsequence if necessary we may assume that $$\widetilde J_{\theta,\epsilon}:=\lim J_{\theta+\epsilon_n}$$ exists. Notice that a priory $\widetilde J_{\theta,\epsilon}$ does not need to have an empty interior. Observe that \begin{equation}\label{EqInclPreim}\partial\widetilde J_{\theta,\epsilon}\supset L_\sigma^{-1}(J_\theta)\;\;\text{and}\;\;\widetilde J_{\theta,\epsilon}\subset L_\sigma^{-1}(K_\theta).\end{equation} Indeed, from \eqref{EqRelationsJulias} we have:
$$\widetilde J_{\theta,\epsilon}\supset J_{\theta,\sigma}\supset L_\sigma^{-1}(J_\theta).$$ On the other hand, arbitrarily close to each point of $L_\sigma^{-1}(J_\theta)$ there is a disk $D$ such that  $$f_{\theta+\epsilon_k}^{N_k}(D)\cap K_\theta=\varnothing$$ for sufficiently large $k$. Thus, $$L_\sigma^{-1}(J_\theta)\subset \overline{K_\theta\setminus \widetilde J_{\theta,\epsilon}},$$ which proves \eqref{EqInclPreim}.

Now, restrict our attention to the case $\sigma\in\mathbb R$ and $\epsilon_n\in\mathbb R$.
Fix $\sigma_0\in\mathbb R$. Let $\zeta_0\in L_{\sigma_0}(K_\theta)\cap J_\theta$ be such that $f_\theta^m(\zeta_0)=0$ for some $m\in\mathbb N$. Let $z_0\in \partial\widetilde J_{\theta,\epsilon}$ be such that $L_{\sigma_0}(z_0)=\zeta_0$. Without loss of generality we may assume that
$$DL_{\sigma_0}(z_0)\neq 0\;\;\text{and}\;\;\frac{dL_{\sigma_0+\delta}(z_0)}{d\delta}\neq 0\;\;\text{at}\;\;\delta=0.$$

Assume for simplicity that $\theta=p/q$, ($p\in\mathbb Z,q\in\mathbb N$, $p,q$ are co-prime) with $q\geqslant 3$.  Notice that $f_\theta^m$ sends conformally a neigborhood of  $\zeta_0$ onto a neighborhood of $0$. For each repelling direction $\nu_r^i$ of $f_\theta$ let $\gamma_i$ be the external ray at $\zeta_0$ such that $f_\theta^m(\gamma_i)$ is tangent to $\nu_r^i$. For sufficiently small $\delta_0$ there exists an inverse branch $\phi_\delta$ of $L_{\sigma_0+\delta}$ defined on $U_{\delta_0}(\zeta_0)$ for $|\delta|<\delta_0$ such that $\phi_\delta$ depends analytically on $\delta$. 
Then there exists $0<\delta_1<\delta_0$ such that for $\delta\neq 0,|\delta|<\delta_1$, one has $$\phi_\delta(\cup\gamma_i\cap U_\delta(\zeta_0))\cap \phi_0(J_\theta\cap U_\delta(\zeta_0))\neq\varnothing.$$ Using \eqref{EqInclPreim} we obtain that $\widetilde J_{\theta,\epsilon}\notin \mathcal J(\sigma_0+\delta)$. It follows that $$\mathcal J(\sigma_0)\cap \mathcal J(\sigma_0+\delta)=\varnothing.$$ Since this is true for any $\sigma_0$ the statement of the proposition follows. For $q=1$ and $q=2$ the proof is similar and we leave it to the reader as an exercise.
%
\end{proof}
\noindent
An immediate consequence of Proposition \ref{PropDistLim} is the following:
\begin{Co}\label{CoDistLim} For every rational number $\theta$ there exists a continuum $\Upsilon$ of sequences $\gamma=\{\gamma_i\}_{i\in\mathbb N}$ of Cremer parameters such that $\gamma_i\to \theta$, the limit $\lim J_{\gamma_i}$ exist for every $\gamma\in\Upsilon$ and $\lim J_{\gamma_i}$ are pairwise distinct for $\gamma\in E$.
\end{Co}
\section{Constructing Cremer Julia sets of high complexity}
\label{sec:proof}
First, let us prove an auxiliary technical statement.
\begin{Prop}\label{PropCremerSeq} For any oracle TM $M^\phi$, any $\theta_0\in\mathbb T$, any $\epsilon_0>0$ and any $n_0\in\mathbb N$ there exists a Cremer parameter $\theta_1\in U_{\epsilon_0}(\theta_0)$, a number $\epsilon_1>0$ and an integer $n_1>n_0$ such that the following conditions are satisfied:
\begin{itemize} \item[$1)$] $\overline{U_{\epsilon_1}(\theta_1})\subset U_{\epsilon_0}(\theta_0)$;
\item[$2)$] for every $\theta\in U_{\epsilon_1}(\theta_1)$ the polynomial $f_\theta$ has a non-zero periodic point in $U_{2^{-{n_1}}}(0)$;
\item[$3)$] for every $\theta\in U_{\epsilon_1}(\theta_1)$ the Turing Machine $M(n_1)$ does not produce a correct $2^{-{n_1}}$-approximation of $J_\theta$ in time $t(n_1)$.
\end{itemize}
\end{Prop}
\begin{proof}[Proof of Proposition \ref{PropCremerSeq}]
Since rational numbers are dense in $\mathbb T$ without loss of generality we may assume that $\theta_0$ is rational. Corollary \ref{CoDistLim} implies that there exist two sequences of Cremer parameters $\{\gamma^1_j\},\{\gamma^2_j\}$  convergent to $\theta_0$ such that the limits of the Julia sets $$J^s:=\lim\limits_{j\to\infty} J_{\gamma^s_j},\;\; s=1,2,$$ exist with respect to the Hausdorff metric and are distinct. Choose $N$ such that $d_H(J^1,J^2)>2^{-N}$. Set $n_{1}=\max\{N,n_0+1\}$. Choose $j$ sufficiently large so that
\begin{itemize}
\item{} $\gamma^1_j,\gamma^2_j\in U_{\epsilon_0}(\theta_0)$;
\item{} $d_H(J_{\gamma^1_j},J_{\gamma^2_j})>2^{1-N}$;
 \item{} the first $t(n_{0})$ dyadic digits of $\gamma^1_j$ and $\gamma^2_j$ coincide.
 \end{itemize}
 There are two possibilities.
 \vskip 0.2cm
 \noindent
 $a)$ The Turing Machine $M$ given the input $n_{1}$ and an oracle for $\theta=\gamma^s_j$ for $s=1$ or $2$ either runs for longer than $t(n_{1})$ time units or does not produce a finite set of complex numbers. Then we set $\theta_{1}=\gamma^s_j$.
 \vskip 0.2cm
 \noindent
 $b)$ Otherwise, the Turing Machine $M$ with an input $n_{1}$ in time $t(n_{1})$ is not able to distinguish between $\gamma^1_j$ and $\gamma^2_j$, therefore produces the same collection of complex points $S$ for these two values. Choose $s\in\{1,2\}$ such that $S$  does not coincide with a $2^{-n_{1}}$ approximation of $J_{\gamma_j^s}$. Set $\theta_{1}=\gamma^s_j$.

 Further, since the correspondence $$\theta\mapsto J_\theta$$ is continuous at Cremer parameters with respect to the Hausdorff metric \cite{Douady-94}, for sufficiently small $\epsilon_{1}$ for any $\theta\in U_{\epsilon_{1}}(\theta_{1})$ the Turing Machine $M_{1}(n_{1})$ does not produce a correct $2^{-n_{1}}$ approximation of $J_\theta$. By the Small Cycles Property, $f_{\theta_{1}}$ has a periodic cycle in the punctured $2^{-n_{1}}$-neighborhood of the origin. From the Implicit Function Theorem it follows that for sufficiently small  $\epsilon_{1}$ the condition $2)$ holds. Finally, to satisfy $1)$ we make $\epsilon_{1}$ smaller if necessary.
\end{proof}
Now we a ready to prove Theorem \ref{ThMain}.
Let $\mathfrak M$ be the set of all Turing Machines. Notice that $\mathfrak M$ is countable. Fix a sequence $\{M_i\}_{i\geqslant 1}$ of Turing Machines such that each $M\in\mathfrak M$ appears in this sequence infinitely many times. Using Proposition \ref{PropCremerSeq} we construct a countable collection $\Omega_1$ of triples $(\theta_1,\epsilon_1,n_1)$, where $\theta_1\in\mathbb T$ is rational, $\epsilon_1>0$ and $n_1\in\mathbb N$, such that
\begin{itemize}
 \item[$1)$] the sets $U_{\epsilon_1}(\theta_1)$ are pairwise disjoint for $(\theta_1,\epsilon_1,n_1)\in\Omega_1$ and their union is dense in $\mathbb T$;
 \item[$2)$] given $(\theta_1,\epsilon_1,n_1)\in\Omega_1$ for every $\theta\in U_{\epsilon_1}(\theta_1)$  the polynomial $f_\theta$ has a non-zero periodic point in $U_{2^{-{n_1}}}(0)$ and the Turing Machine $M_1(n_1)$ does not produce a correct $2^{-{n_1}}$-approximation of $J_\theta$ in time $t(n_1)$.
\end{itemize}
Moreover, using Proposition \ref{PropCremerSeq} by induction we construct a sequence $\{\Omega_i\}$, where $\Omega_i$ is a collection of triples $(\theta_i,\epsilon_i,n_i)$ satisfying  the conditions for $\Omega_1$ with $M_1$ replaced by $M_i$ and, in addition, the following conditions:
\begin{itemize}
\item[$3)$] for $i\in\mathbb N$ if $(\theta_i,\epsilon_i,n_i)\in\Omega_i$ and $(\theta_{i+1},\epsilon_{i+1},n_{i+1})\in\Omega_{i+1}$ then either $U_{\epsilon_i}(\theta_i)\cap U_{\epsilon_{i+1}}(\theta_{i+1})=\varnothing$ or $U_{\epsilon_i}(\theta_i)\supset U_{\epsilon_{i+1}}(\theta_{i+1})$;
\item[$4)$] if $U_{\epsilon_i}(\theta_i)\supset U_{\epsilon_{i+1}}(\theta_{i+1})$ then $n_{i+1}>n_i$.
\end{itemize}

Let $A_i$ be the union of the sets $U_{\epsilon_i}(\theta_i)$ over triples $(\theta_i,\epsilon_i,n_i)\in\Omega_i$. Then $A=\cap A_i$ is a dense $G_\delta$ subset of $\mathbb T$. Let $\theta_\infty\in A$. Then for every $i\in\mathbb N$ there exists a unique  $(\theta_i,\epsilon_i,n_i)\in\Omega_i$ such that $\theta_\infty\in U_{\epsilon_i}(\theta_i)$. By the condition $2)$, the polynomial $f_{\theta_\infty}$ has a small cycle property, therefore, $\theta_\infty$ is a Cremer parameter. Moreover, for every Turing Machine $M$ there exists infinitely many positive integers $n$ such that $M$ does not produce a correct $2^{-n}$-approximation of $J_{\theta_\infty}$ in time $t(n)$. This finishes the proof of Theorem \ref{ThMain}.


\begin{thebibliography}{}


\bibitem{BBY}  I. Binder, M. Braverman, and M. Yampolsky. \emph{Filled Julia sets with empty interior are computable}, Journal FoCM, 7(2007), 405-416.

\bibitem{BY06} M. Braverman and M. Yampolsky,  {\it Non-computable Julia sets}, Journ. Amer. Math. Soc., 19 (2006), 551-578.

\bibitem{BY08}
\newblock M. Braverman and M. Yampolsky,
\newblock \emph{Computability of Julia Sets},
\newblock Algorithms and Computation in Mathematics, 23, Springer-Verlag, Berlin, 2009.

\bibitem{BY09} M. Braverman and M. Yampolsky, {\it Constructing Locally Connected Non-Computable Julia Sets,} Commun. Mah. Phys., 291(2009), p. 513-532

\bibitem{Brjuno-71} A. D. Brjuno, \emph{Analytic form of differential equations. I, II,} Trudy Moskov. Mat. Obsh. {\bf 25} (1971), 119-262; ibid. {\bf 26} (1972), 199-239.


\bibitem{Douady-94}
A. Douady, \emph{Does a Julia set depend continuously on the polynomial?} Complex dynamical systems (Cincinnati, OH, 1994), 91-138, Proc. Sympos. Appl. Math., {\bf 49}, AMS Short Course Lecture Notes, Amer. Math. Soc., Providence, RI, 1994.

\bibitem{Lav} P. Lavaurs, \emph{Syst{\`e}mes dynamiques holomorphes: explosion de points p{\'e}riodiques
paraboliques.} These, Universit{\'e} Paris-Sud, 1989

\bibitem{Mil} J. Milnor, \emph{Dynamics in one complex variable, 3rd ed.}, Princeton University Press, 2006.

\bibitem{Pap} C. M. Papadimitriou, {\it Computational complexity}, Addision-Wesley, Reading, Massachusetts, 1994.


\bibitem{Sip} M. Sipser, {\it Introduction to the theory of computation, second edition}, BWS Publishing Company, Boston, 2005.


\bibitem{Tur} Turing, A. M., {\it On Computable Numbers, With an Application to the {E}ntscheidungsproblem},{\it Proc. London Math. Soc.,} 1936, pp. 230-265.

\bibitem{Yoccoz-95} J.C. Yoccoz, \emph{Th\'eor\`eme de Siegel, nombres de Bruno et polyn\^omes quadratiques}, Petits diviseurs en dimension 1, Asterisque {\bf 231} (1995).

\bibitem{Zin} M. Zinsmeister, \emph{Basic Parabolic Implosion in Five Days. Course given at Jyvaskyla,  1997}, available at {\sl http://www.univ-orleans.fr/mapmo/membres/zins/articles/jyv4.ps}

\end{thebibliography}
\end{document}